\documentclass{amsart}
\usepackage[english]{babel}
\usepackage{latexsym,amsfonts}
\usepackage{amssymb}
\usepackage{amsmath}
\usepackage{amsthm}

\numberwithin{equation}{section}

\newtheorem{lemma}{Lemma}[section]
\newtheorem{theorem}[lemma]{Theorem}
\newtheorem{proposition}[lemma]{Proposition}
\newtheorem{corollary}[lemma]{Corollary}

\theoremstyle{remark}

\newtheorem{ex}[lemma]{Example}

\theoremstyle{definition}

\def\U{\mathrm{U}}
\def\GL{\mathrm{GL}}  
\def\Or{\mathrm{O}}  
\def\AGL{\mathrm{AGL}}           
\def\char{{\rm char\,}}
\def\Ker{{\rm Ker\,}}
\def\diag{\mathrm{diag}}

\def\T{\mathop{\rm T}\nolimits}

\newcommand{\F}{\mathbb{F}}    
    
 \newcommand{\C}{\mathbf{C}}    
     
\newcommand{\Tr}{\mathcal{T}}    
   
 \newcommand{\Z}{\mathbf{Z}}

\begin{document}

\title[Regular subgroups of the affine group with no translations]{Regular subgroups of the affine group \\ with no 
translations}

\author{M.A. Pellegrini}
\email{marcoantonio.pellegrini@unicatt.it}

\author{M.C. Tamburini Bellani}
\email{mariaclara.tamburini@gmail.com}

\address{Dipartimento di Matematica e Fisica, Universit\`a Cattolica del Sacro Cuore,\\
Via Musei 41, I-25121 Brescia, Italy}

\keywords{Regular subgroup,  affine group, translations.}
\subjclass[2010]{20B35, 15A63, 15A21.}

\begin{abstract}
Given a regular subgroup $R$ of $\AGL_n(\F)$, one can ask if $R$ contains nontrivial translations.
A negative answer to this question was given by Liebeck, Praeger and Saxl for $\AGL_2(p)$ ($p$ a prime), $\AGL_3(p)$ 
($p$ 
odd) and for 
$\AGL_4(2)$. A positive answer was given by Heged\H{u}s for $\AGL_n(p)$ when $n\geq 4$ if $p$ is odd and for 
$n=3$ or $n\geq 5$ if $p=2$. 
A first generalization to finite fields of Heged\H{u}s' construction was recently obtained by Catino, Colazzo and 
Stefanelli.
In this paper we give examples of such subgroups in $\AGL_n(\F)$ for any $n \geq  5$ and any
field $\F$. 
For $n < 5$ we provide necessary and sufficient conditions for their existence, assuming $R$ to be unipotent if $\char 
\F=0$.
\end{abstract}

\maketitle

\section{Introduction}

Consider the affine group 
$$\AGL_n(\F)=\left\{\begin{pmatrix}
1 & v \\ 0 & A  \end{pmatrix} : v \in \F^n,\; A \in \GL_n(\F) 
 \right\}$$
acting on the right on the row vector space $\F^{n+1}$,
whose canonical basis will be denoted by $\{e_0,e_1,\ldots,e_n\}$.
Furthermore, denote by
$$\pi: \AGL_n(\F) \rightarrow \GL_n(\F)$$
the obvious epimorphism $\begin{pmatrix}
1 & v \\ 0 & A  \end{pmatrix} \mapsto A$, whose kernel is the translation
group $\Tr$.
A subgroup $R$ of $\AGL_n(\F)$ is called \emph{regular} if it acts
regularly on the set of the affine points: namely if, for every $v\in \F^n$, there exists a unique element of $R$ having the affine point $(1,v)$ as first row.

The problem of the existence of regular subgroups of $\AGL_n(\F)$ having no translations
other than the identity was first raised by Liebeck, Praeger and Saxl in \cite{LPS}.
Clearly, for such subgroups we have
\begin{equation}\label{piR}
R\cong \pi(R). 
\end{equation}
In the case of fields $\F=\F_p$ of prime order $p$, the above-mentioned authors also proved that no such regular subgroups exist for $\AGL_2(p)$, any $p$, $\AGL_3(p)$, $p>2$, and for $\AGL_4(2)$.
The  first positive examples, which proved their existence, were constructed by 
Heged\H{u}s in \cite{H2000}, in the case $\F=\F_p$. 
More precisely, he proved that $\AGL_n(p)$ contains a regular subgroup having no translations other than  the identity, whenever (\emph{i})  $n=3$ or $n\geq 5$, if $p=2$ or (\emph{ii}) $n\geq 4$, if $p>2$.
The crucial property
that he used is the existence of a non-degenerate quadratic form over $\F_p^{n-1}$ and an embedding
of the additive group $(\F_p,+)$ into the corresponding orthogonal group. 
Clearly this property holds for much more general fields than $\F_p$. This fact was recently  used
by Catino, Colazzo and  Stefanelli who extended Heged\H{u}s' result to $\F=\F_{p^\ell}$, \cite{CCS2}.

In Section \ref{neg} we extend to an arbitrary field $\F$ the negative results of \cite{LPS}, giving an independent proof of the following facts:

\begin{theorem}\label{negm}
Let $R$ be a regular subgroup of $\AGL_n(\F)$ and suppose that $R$ is unipotent if $\char \F=0$.
Assume that one of the following conditions holds:
\begin{itemize}
\item[\emph{(i)}] $n\leq 2$;
\item[\emph{(ii)}] $n=3$ and  $\F\neq \F_2$; 
\item[\emph{(iii)}] $n=4$ and $\char \F=2$.
\end{itemize}
Then $R$ contains a nontrivial translation.
\end{theorem}

On the other hand, using  Heged\H{u}s' method, we generalize \cite{CCS2} proving the following result. 

\begin{theorem}\label{main} 
Let  $\F$ be any field  and $W$ be a subspace of $\F$, viewed as a vector space over its prime field $\F_0$.
Assume that one of the following conditions holds:
\begin{itemize}
\item[\emph{(i)}] $n=3$ and  $\F=\F_2$;
\item[\emph{(ii)}] $n\ge 4$ and  $\char \F\ne 2$;
\item[\emph{(iii)}] $n\geq 5$ and $\char \F=2$.
\end{itemize}
Then, there exists a regular subgroup $R_W$ of $\AGL_n(\F)$ such that $R_W\cap \Tr\cong (W,+)$. 
In particular there exists a regular subgroup $R_{\{0\}}$ such that $R_{\{0\}} 
\cap \Tr=\{I_{n+1}\}$.
\end{theorem}

Generalizations of this result to other additive subgroups $W$ of $\F^n$ can easily be
obtained taking, for example, direct products of affine groups or exploiting the flexibility of our Lemma \ref{lem1}.

We point out that the present paper provides the first known examples of regular subgroups of $\AGL_n(\F)$, $n\geq 4$, 
intersecting trivially the subgroup $\Tr$ in the case of infinite fields $\F$.

Combining Theorems \ref{negm} and  \ref{main}  we have the following characterization.

\begin{corollary}
Let $\F$ be a field of positive characteristic. Then the affine group $\AGL_n(\F)$ contains regular subgroups intersecting trivially $\Tr$ if and only if one of the 
following occurs:
\begin{itemize}
 \item[\emph{(i)}] $n=3$ and $|\F|=2$;
 \item[\emph{(ii)}] $n=4$ and $\char \F> 2$;
 \item[\emph{(iii)}] $n\geq 5$.
\end{itemize}
\end{corollary}

\begin{corollary}
Let $\F$ be a field of characteristic $0$. Then the affine group $\AGL_n(\F)$ contains unipotent regular subgroups intersecting trivially $\Tr$ if and only if $n\geq 4$.
\end{corollary}

\section{Non existence results}\label{neg}

We recall that every abelian regular subgroup of $\AGL_n(\F)$ contains nontrivial translations, see 
\cite[Lemma 5.2]{Ischia}. 
Furthermore, by \cite[Theorem 3.2]{Ischia}, if $\char \F> 0$ any regular subgroup $R$ of $\AGL_n(\F)$ is unipotent. 
We will make repeated use of the fact that, up to conjugation under $\AGL_n(\F)$, any unipotent subgroup of $\AGL_n(\F)$ is contained
in the subgroup $\U_{n+1}(\F)$ of upper unitriangular matrices of $\GL_{n+1}(\F)$, see \cite[17.5]{H}.

\begin{lemma}\label{2}
Every unipotent regular subgroup $R$ of $\AGL_2(\F)$ contains nontrivial translations. 
\end{lemma}

\begin{proof}
If the claim is false, then $R\cong \pi(R)$  by \eqref{piR}. Now, $\pi(R)$ is a subgroup of $ \U_2(\F)$, which is abelian. So, $R$ would be abelian, whence the contradiction $R\cap \Tr\neq\{I_3\}$.
\end{proof}

In the following we denote by $E_{i,j}$ the elementary matrix having $1$ at position $(i,j)$, $0$ elsewhere. Also, $J_m=I_m+\sum_{i=1}^{m-1} E_{i,i+1}$ denotes the unipotent upper triangular Jordan block of size $m$.

\begin{proposition}\label{3}
The group $\AGL_3(\F)$ contains a unipotent regular subgroup $R$ such that 
$R\cap \Tr=\{I_{4}\}$ if and only if $|\F|=2$.
\end{proposition}

\begin{proof}
Suppose that $R\cap \Tr=\{I_4\}$.
By what observed at the beginning of this section  $R$ is not abelian.
Moreover $R\cong \pi(R)$ by \eqref{piR}.
By \cite[Theorem 4.4]{PT}, up to conjugation under $\AGL_3(\F)$, we may suppose that $\Z(R)$ contains an element $\tilde z$ which is one of the following Jordan forms:
$J_4$, $\diag(J_3,J_1)$, $\diag(J_2,J_2)$ or $\diag(J_2,J_1,J_1)$.
Since $\C_{\AGL_3(\F)}(J_4)$ is abelian, $\tilde z\neq J_4$. Similarly, every unipotent subgroup of $\pi(\C_{\AGL_3(\F)}(\diag(J_3,J_1)))$ is abelian: so, $\tilde z\neq \diag(J_3,J_1)$.
Clearly $\tilde z$ cannot be the translation $\diag(J_2,J_1,J_1)$.

Thus, we are left to consider the case where $\tilde z=\diag(J_2,J_2)$. Replacing $\tilde z$ with its conjugate by the permutation matrix corresponding to the transposition $(2,3)$, we may suppose
that 
$$z=\begin{pmatrix}
1 & 0 & 1 & 0  \\ 0 & 1 & 0 & 1 \\ 0 & 0 & 1 & 0 \\ 0 & 0& 0 & 1
  \end{pmatrix}\in \Z(R).
$$
In this way, $\C_{\AGL_3(\F)}(z)$ is upper triangular, and in particular
$R\leq \C_{\AGL_3(\F)}(z)$ consists of matrices of shape
$$r_{(x_1,x_2,x_3)}=\begin{pmatrix}
1 & x_1 & x_2 & x_3 \\
0 & 1 & a(X) & b(X)\\
0 & 0 & 1 & x_1\\
0 & 0 & 0 & 1
\end{pmatrix},\qquad X=(x_1,x_2,x_3).
$$
Now, for all $x\in \F$ set $t_{x}=r_{(0,0,x)}$ and call
$\tilde a(x)=a(0,0,x)$ and $\tilde b(x)=b(0,0,x)$.
It is easy to see that $t_{x}t_{y}=t_{x+y}$, which implies that $\tilde b(x+y)=\tilde b(x)+\tilde b(y)$ for all $x,y\in\F$.
Suppose that $\tilde a(x)=0$ for some $x\in \F$. Then $[r_{(1,0,0)},t_x]=I_4+\tilde b(x) E_{1,4}\in R \cap \Tr$, which implies $\tilde b(x)=0$. In particular, $t_x\in \Tr$ and so $x=0$. Hence, for all $x\neq 0$, we have $\tilde a(x)\neq 0$ and we can consider the element $r_{(\tilde a(x)^{-1},0,0)}$. We have 
$$[r_{(\tilde a(x)^{-1},0,0)},t_x] z = I_4+2 E_{1,3}+ \frac{\tilde b(x)+1}{\tilde a(x)}E_{1,4}.$$
If $\char \F\neq 2$, this element is a nontrivial translation and our claim is proved.
So, assume $\char \F=2$. Then $\tilde b(x)=1$ for all $x\neq 0$.

If $|\F|>2$,  we can take an element $w\in \F$ such that $w\neq 0,1$.
From the additivity of $\tilde b(x)$ we obtain the contradiction 
$1+1=\tilde b(1)+\tilde b(w)=\tilde b(1+w)=1$.
Finally, let $\F=\F_2$. The reader can verify that 
$$R=\langle  (I_4+E_{1,2}+E_{2,3}+E_{3,4}) , (I_4+E_{1,4}+E_{2,3}+E_{2,4}) \rangle$$
is a regular 
subgroup of $\AGL_3(2)$ such that $R\cap \Tr=\{I_4\}$.
\end{proof}

\begin{proposition}\label{4}
Let $\F$ be a field of characteristic $2$. Then every regular subgroup $R$ of $\AGL_4(\F)$ contains nontrivial translations.
\end{proposition}

\begin{proof}
By what observed at the beginning of this section, $R$ is unipotent since $\char \F=2$. By \cite[Theorem 4.4]{PT}, up to conjugation in $\AGL_4(\F)$, we may suppose that  $\Z(R)$ contains an element $\tilde z$ which is one of the following Jordan forms:
$J_5$, $\diag(J_4,J_1)$, $\diag(J_3,J_2)$, $\diag(J_3,J_1,J_1)$, $\diag(J_2,J_2,J_1)$ or $\diag(J_2,J_1,J_1,J_1)$.
Suppose that $R\cap \Tr=\{I_5\}$, whence $R\cong \pi(R)$. It follows $r^4=I_5$ for all $r\in R$, which gives $\tilde z\neq J_5$. Furthermore, $\tilde z\not \in \{\diag(J_3,J_2),\diag(J_3,J_1,J_1),\diag(J_2,J_1,$ $J_1,J_1)\}$, since
$\diag(J_3,J_2)^2=\diag(J_3,J_1,J_1)^2=I_5+E_{1,3}\in \Tr$ and  $\diag(J_2,J_1,$ $J_1,J_1)\in \Tr$.
Now, if $\tilde z=\diag(J_4,J_1)$, then every unipotent subgroup of  $\pi(\C_{\AGL_4(\F)}(\tilde z))$ is abelian, which implies that $R$ is abelian, an absurd. We are left to consider the case $\tilde z=\diag(J_2,J_2,J_1)$. 
Using the fact that $R$ is unipotent we obtain that
$$R\leq \left\{\begin{pmatrix}
1 & \rho_1 & \rho_2 & \rho_3 & \rho_4\\
0 & 1 & 0 & \rho_2 & 0 \\
0 & \rho_5 & 1 & \rho_6 & \rho_7\\
0 & 0 & 0 & 1 & 0\\
0 & \rho_8 & 0 & \rho_9 & 1\end{pmatrix}: \rho_i\in \F\right\}.
$$
Conjugating by the permutation matrix corresponding to $(2,4,5,3)$ we may suppose that
$z=I_5+E_{1,4}+E_{2,5}\in \Z(R)$ and write any element $r$ of $R$ as
\begin{equation}\label{r1}
r=\begin{pmatrix}
1 & x_1 & x_2 & x_3 & x_4\\
0 & 1 & \alpha_1(X) & \alpha_2(X) & \alpha_3(X)\\
0 & 0 & 1 & \alpha_4(X) & \alpha_5(X)\\
0 & 0 & 0 & 1 & x_1\\
0 & 0 & 0 & 0 & 1
        \end{pmatrix},\qquad X=(x_1,x_2,x_3,x_4).
\end{equation}
So, $R\cong \pi(R)\leq \U_4(\F)$, which implies that $[[r_1,r_2],r_3]\in \Z(R)$ for all $r_1,r_2,r_3\in R$.

It will be convenient to call $s$ the element of $R$ having $(1,1,0,0,0)$ as first row. Since $g=I_5+\alpha_5(1,0,0,0) E_{3,4}$ centralizes $z$, the elements of $R^g$ still have shape \eqref{r1} and $g^{-1}sg$ still has $(1,1,0,0,0)$ as first row, but with the component of position $(3,5)$ equal to $0$.
So, up to conjugation by $g$, we may suppose that
\begin{equation}\label{s}
 z=\begin{pmatrix}
1 & 0 & 0 & 1 & 0\\
0 & 1 & 0 & 0 & 1\\
0 & 0 & 1 & 0 & 0\\
0 & 0 & 0 & 1 & 0\\
0 & 0 & 0 & 0 & 1
 \end{pmatrix}\in \Z(R), \qquad
s=\begin{pmatrix}
1 & 1 & 0 & 0 & 0\\
0 & 1 & \beta_1 & \beta_2 & \beta_3\\
0 & 0 & 1 & \beta_4 & 0\\
0 & 0 & 0 & 1 & 1\\
0 & 0 & 0 & 0 & 1
 \end{pmatrix}.
 \end{equation}

Now, suppose  that $\alpha_1(X)=0$ for all $X \in \F^4$ (and in particular $\beta_1=0$).
In this case the stabilizer $R_0$ in $R$ of $V_0=\langle e_0,e_1,e_3,e_4\rangle $,
which consists of the matrices having $x_2=0$, induces a regular subgroup
$\tilde R$ of $\AGL_3(\F)$. In particular, since $R_0\cap \Tr=\{I_5\}$,  $\tilde R$ contains no translations other than the identity: by Proposition \ref{3} we must have $\F=\F_2$.
Actually, by \cite[Lemma 7.2]{LPS}, we may assume 
$\F\ne \F_2$ from the beginning. Nevertheless, to keep our proof independent, we exclude 
this possibility directly.
So, assume $\F=\F_2$ and take the two following elements of $R$:
$$v=\begin{pmatrix}
1 & 0 & 1 & 0 & 0\\
0 & 1 & 0 & \nu_2 & \nu_3\\
0 & 0 & 1 & \nu_4 & \nu_5\\
0 & 0 & 0 & 1 & 0\\
0 & 0 & 0 & 0 & 1
 \end{pmatrix},\quad t=\begin{pmatrix}
1 & 0 & 0 & 0 & 1\\
0 & 1 & 0 & \xi_2 & \xi_3\\
0 & 0 & 1 & \xi_4 & \xi_5\\
0 & 0 & 0 & 1 & 0\\
0 & 0 & 0 & 0 & 1
 \end{pmatrix}\in R.
 $$   
From $v^2=I_5+\nu_4 E_{1,4}+\nu_5 E_{1,5}$ and $[v,t]=I_5+\xi_4E_{1,4}+\xi_5E_{1,5}$
we obtain $\nu_4=\nu_5=\xi_4=\xi_5=0$.
Suppose that $\xi_2=0$. Then $[s,t]=I_5+\xi_3E_{1,5}\in \Tr$, which implies $\xi_3=0$ and then the absurd $t\in \Tr$. Hence, $\xi_2=1$ and $z[s,t]=I_5+(1+\xi_3)E_{1,5}$ gives $\xi_3=1$.
Now, from $v,vz,vt,vtz \not \in \Tr$ we get, respectively, $(\nu_2,\nu_3)
\neq (0,0),(0,1),(1,1),(1,0)$, an absurd.

This proves that there exists an element 
$$u=\begin{pmatrix}
1 & y_1 & y_2 & y_3 & y_4\\
0 & 1 & \gamma_1 & \gamma_2 & \gamma_3\\
0 & 0 & 1 & \gamma_4 & \gamma_5\\
0 & 0 & 0 & 1 & y_1\\
0 & 0 & 0 & 0 & 1
\end{pmatrix}\in R$$ having $\gamma_1\neq 0$. Taking $s$ as in \eqref{s}, from $s[[s,u],s]=[[s,u],s]s$, we get $\beta_4=\frac{\beta_1\gamma_4}{\gamma_1}$. 
Suppose $\gamma_4\neq 0$. In this case, $s^4=I_5+\beta_1\beta_4E_{1,5}=I_5+\frac{\beta_1^2\gamma_4}{\gamma_1}E_{1,5}$ implies $\beta_1=0$, whence the contradiction $[s,u]^2=I_5+\gamma_1\gamma_4E_{1,5}\neq I_5$.
So, $\gamma_4=\beta_4=0$, which implies that $[[s,u],u]=I_5+\gamma_1\gamma_5E_{1,5}$ and then $\gamma_5=0$. Now, taking any $r\in R$ (having shape \eqref{r1}), the condition
$s[[s,r],u]=[[s,r],u]s$ implies $\alpha_4(X)=0$ for all $X\in \F^4$. Hence,
$[[u,r],s]=I_5+\gamma_1\alpha_5(X)$ gives $\alpha_5(X)=0$ for all $X\in \F^4$.
So, 
every element $r$ of $R$ has shape
\begin{equation}\label{R}
r=\begin{pmatrix}
1 & x_1 & x_2 & x_3 & x_4\\
0 & 1 & \alpha_1(X) & \alpha_2(X) & \alpha_3(X)\\
0 & 0 & 1 & 0 & 0\\
0 & 0 & 0 & 1 & x_1\\
0 & 0 & 0 & 0 & 1
        \end{pmatrix},\qquad  X=X(x_1,x_2,x_3,x_4).
\end{equation}
Now, set $m_{x}=r_{(x,0,0,0)}$ and $t_{(x_2,x_4)}=r_{(0,x_2,0,x_4)}$ and call $\tilde \alpha_i(x_2,x_4)=\alpha_i(0,x_2,$ $0,x_4)$ for any $i=1,2,3$.
Note that $t_{(x_2,x_4)}t_{(y_2,y_4)}=t_{(x_2+x_4,y_2+y_4)}$, which implies $\tilde \alpha_i(x_2+y_2,x_4+y_4)=\tilde \alpha_i(x_2,x_4)\tilde \alpha_i(y_2,y_4)$ ($i=1,2,3$) for all $x_2,x_4,y_2,y_4\in \F$.
Suppose that $\tilde \alpha_2(x_2,x_4)=0$ for some $x_2,x_4\in \F$. Then
$[s,t_{(x_2,x_4)}]=I_5+ \tilde\alpha_1(x_2,x_4)E_{1,3}+\tilde \alpha_3(x_2,x_4) E_{1,5}\in \Tr$, where $s$ is as in \eqref{s} with $\beta_4=0$.
This implies $\tilde \alpha_1(x_2,x_4)=\tilde \alpha_3(x_2,x_4)=0$. In particular, we obtain that 
$t_{(x_2,x_4)}\in \Tr$, giving $x_2=x_4=0$. We conclude that, if $(x_2,x_4)\neq (0,0)$, then $\tilde \alpha_2(x_2,x_4)\neq 0$ and so we can consider
the element $m_{\tilde \alpha_2(x_2,x_4)^{-1}}$.
We have
$$z[m_{\tilde \alpha_2(x_2,x_4)^{-1}},t_{(x_2,x_4)} ]=I_5+\frac{\tilde \alpha_1(x_2,x_4)}{\tilde \alpha_2(x_2,x_4)}E_{1,3}+\frac{\tilde \alpha_3(x_2,x_4)+1}{\tilde \alpha_2(x_2,x_4)}E_{1,5},$$
which gives $\tilde \alpha_1(x_2,x_4)=0$ and $\tilde \alpha_3(x_2,x_4)=1$ for all $(x_2,x_4)\neq (0,0)$.
In particular, we have $1+1=\tilde \alpha_3(1,0)+\tilde \alpha_3(0,1)=\tilde \alpha_3(1,1)=1$, an absurd.
\end{proof}

\begin{proof}[Proof of Theorem \ref{negm}]
If $n=1$, then $\AGL_1(\F)=\Tr$. For $n\geq 2$ the statement follows from the fact that, if $\char\F >0$, then any regular subgroup of $\AGL_n(\F)$ is unipotent.
It suffices to apply Lemma \ref{2} and Propositions \ref{3} and \ref{4}.
\end{proof}

\section{Proof of Theorem \ref{main}}

In the following, the group of isometries of a quadratic form $Q$ on $\F^m$ will be denoted by
$\Or_m(\F,Q)$. Namely
$$\Or_m(\F,Q)=\{A \in \GL_m(\F): Q(vA)=Q(v), \;\textrm{for all } v \in \F^m\}.$$

\begin{lemma} \label{lem1}
Let $m,k\geq 1$ and  $d$ be a fixed row vector of $\F^k$.
Let $Q$ be a quadratic form on $\F^m$ with polar form $J$ and 
$\varphi: (\F^k,+)\rightarrow \Or_{m}(\F,Q)$ be a group homomorphism.
Then the following holds:
\begin{itemize}
\item[(\textrm{a})] The set $N=\left\{\begin{pmatrix}
1 & v & Q(v)d\\
0 & I_{m} & Jv^{\T}\otimes d\\
0 & 0 & I_k
\end{pmatrix}:  v\in \F^{m}\right\}$ is a subgroup of $\AGL_{m+k}(\F)$. 

\item[(\rm{b})] The set
$M=\left\{\begin{pmatrix}
1&0 &a\\
0&\varphi(a)&0\\
0&0&I_k
\end{pmatrix} : a\in \F^k\right\}$ is a subgroup of $\AGL_{m+k}(\F)$ that normalizes $N$;

\item[(\rm{c})] The group $R=M\ltimes N$ is a regular subgroup of 
 $\AGL_{m+k}(\F)$ and, if $Q$ is non-degenerate and $d\neq 0$,  we have $R\cap \Tr\cong \Ker(\varphi)$.
\end{itemize}
\end{lemma}

\begin{proof}
Items (a) and the first part of item (b) follow from easy calculations. We just show that $M$ normalizes $N$:
$$\begin{pmatrix}
1&0&-a\\
0&\varphi(a)^{-1}&0\\
0&0&I_k
\end{pmatrix}\begin{pmatrix}
1&v&Q(v)d\\
0&I_{m}& J v^{\T}\otimes d\\
0&0&I_k
\end{pmatrix}\begin{pmatrix}
1&0 &a\\
0&\varphi(a)&0\\
0&0&I_k
\end{pmatrix}=$$
$$=\begin{pmatrix}
1&v\varphi(a)& Q(v)d\\
0&I_{m}&\varphi(a)^{-1}Jv^{\T}\otimes d\\
0&0&I_k
\end{pmatrix}= \begin{pmatrix}
1&v\varphi(a)& Q(v\varphi(a))d\\
0&I_{m} & J\left(v\varphi(a)\right)^{\T}\otimes d\\
0&0&I_k
\end{pmatrix}.$$
To prove (c) we first observe that
\begin{eqnarray*}
R & = & \left\{\begin{pmatrix}
1& v & a + Q(v)d\\
0&\varphi(a) &\varphi(a) J v^{\T}\otimes d\\
0&0&I_k
\end{pmatrix}:  v\in \F^{m}, a \in \F^k \right\}\\
& = &  \left\{\begin{pmatrix}
1& v & a \\
0&\varphi(a-Q(v)d)& \varphi(a-Q(v)d) Jv^{\T}\otimes d\\
0&0& I_k
\end{pmatrix}:  v\in \F^{m}, a \in \F^k \right\}.
\end{eqnarray*}
Now, let $r$ be the element of $R$ having $(1,v,a)$ as first row.
Then $r\in R \cap \Tr$ if and only if $\varphi(a-Q(v)d)=I_{m}$ and $Jv^{\T}\otimes d=0$.
If $Q$ is non-degenerate and $d\neq 0$, from $Jv^{\T}\otimes d=0$  we obtain $v=0$ and so $\varphi(a-Q(v)d)=\varphi(a)=I_{m}$ if and only if $a \in \Ker(\varphi)$.
\end{proof}

\begin{proof}[Proof of Theorem \ref{main}]
In Lemma \ref{lem1} take $(m,k)$ $=(n-2,2)$ if $n$ is even and $\char \F=2$, $(m,k)=(n-1,1)$ otherwise.
Then there exists a 
non-degenerate quadratic form $Q$ on $\F^{m}$ such that $\Or_{m}(\F,Q)$ contains
a subgroup isomorphic to $(\F^k,+)$ (see \cite[Chapter 11]{T} and the examples below). 
Let us fix  an embedding 
$$\psi:  (\F^k,+)\to \Or_{m}(\F,Q).$$ 
Choose any complement $U$ of $W$ in $\F^k$. Since 
$\F^k=U\oplus W$, each $a\in \F^k$ can be written in a unique way as 
$a=u_a+w_a$ with $u_a \in U$, $w_a \in W$. 
Calling $\varpi:\F^k\rightarrow \F^k$ the projection
onto $U$, defined by  $\varpi(a)=u_a$,  
the group homomorphism  $\varphi= \psi\varpi: (\F^k,+)\to \Or_{m}(\F,Q)$ has kernel $W$.
Taking $M$ and $N$ as in Lemma \ref{lem1} we have that  $R=M\ltimes N$ is a regular subgroup of 
$\AGL_n(\F)$ such that  $R\cap \Tr\cong (W,+)$.
\end{proof}

\begin{ex}
Assume $\char \F\ne 2$, $n\ge 4$ and consider the quadratic form $Q$ on $\F^{n-1}$ defined by 
$$Q(x_1, \dots, x_{n-1})=x_1x_3 - x_2^2+\frac{1}{2} \sum_{i=4}^{n-1} x_i^2.$$
The polar form of $Q$ has Gram matrix
$J=\diag(J_3,I_{n-4})\in \GL_{n-1}(\F)$, where   
$J_3=\begin{pmatrix}
0 & 0 & 1\\
0 & -2 & 0\\
1 & 0 & 0
\end{pmatrix}$. Then the application
$\varphi: \F\rightarrow \GL_{n-1}(\F)$ defined by 
$$\varphi(a)=I_{n-1} + 2a E_{2,1} + a^2 E_{3,1} + a E_{3,2}$$
is a 
monomorphism from $(\F, +)$ into $\Or_{n-1}(\F,Q)$. By Lemma \ref{lem1}  
\begin{equation}\label{subg}
R=\left\{\begin{pmatrix}
1& v & a +Q(v)\\
0&\varphi(a)&\varphi(a) J v^{\T}\\
0&0&1
\end{pmatrix}: v\in \F^{n-1}, a \in \F\right\}
\end{equation}
is a regular subgroup of $\AGL_n(\F)$ with no nontrivial translations.
\end{ex}

\begin{ex}
Assume  $\char \F=2$, $n=2t+1\ge 3$  and consider the quadratic form $Q$ on $\F^{2t}$ defined by 
\begin{equation}\label{Q}
Q(x_1, \dots,x_{2t})= \sum_{i=1}^t x_i x_{t+i}.
\end{equation}
The polar form of $Q$ has Gram matrix
$\left( \begin{smallmatrix} 0 & I_t\\ I_t & 0\end{smallmatrix} \right)$. 
If $n\geq 5$, the application $\varphi: \F\rightarrow 
\GL_{2t}(\F)$ defined by $\varphi(a)= I_{2t}+ a (E_{1,t}+E_{2t,t+1})$
is a monomorphism from $(\F, +)$ into $\Or_{2t}(\F,Q)$.
If $n=3$ and $\F=\F_2$ define  the monomorphism $\varphi: \F_2\rightarrow 
\Or_{2}(\F_2)$ by setting $\varphi(1)=E_{1,2}+E_{2,1}$.
With these definitions of $Q$ and $\varphi$, again by Lemma \ref{lem1} we obtain a regular subgroup of $\AGL_n(\F)$ 
of shape \eqref{subg}, intersecting trivially  $\Tr$.
\end{ex}

\begin{ex}\label{ex3}
Assume  $\char \F=2$, $n=2t+2\ge 6$  and consider the quadratic form $Q$ on $\F^{2t}$ defined by \eqref{Q}. The 
application $\varphi: \F^2\rightarrow 
\GL_{2t}(\F)$ defined by 
$$\varphi(a,b)=I_{2t}+ a (E_{1,t}+E_{2t,t+1})+b(E_{1,2t}+E_{t,t+1})+abE_{1,t+1}$$
is a monomorphism from $(\F^2, +)$ into $\Or_{2t}(\F,Q)$.
With these definitions of $Q$ and $\varphi$,  we obtain a regular subgroup of 
$\AGL_n(\F)$ 
of shape \eqref{subg}, intersecting  trivially $\Tr$.
\end{ex}

We conclude observing that direct products of regular subgroups  intersecting trivially $\Tr$ clearly give rise to 
regular subgroups with the same property. However, by point (iii) of Theorem \ref{negm} 
a regular subgroup of $\AGL_6(2^\ell)$, $\ell>1$, with no nontrivial translations  cannot be obtained as a direct 
product of regular subgroups of $\AGL_3(2^\ell)$, as done in \cite{CCS2}. So the inductive argument used there for 
$\AGL_{2m}(2^\ell)$, $2m\geq 6$, has a little gap in the basis of induction. But the claim is correct, actually in a stronger form, as confirmed by our Example \ref{ex3}.


\begin{thebibliography}{10}

\bibitem{CCS2} F. Catino, I. Colazzo and P. Stefanelli, Regular subgroups of the affine group and asymmetric 
product of radical braces, \emph{J. Algebra} \textbf{455} (2016) 164--182.

\bibitem{H2000} P. Heged\H{u}s, Regular subgroups of the affine group, \emph{J. Algebra} \textbf{225} (2000), no. 2, 
740--742.


\bibitem{H} J.E. Humphreys, \emph{Linear algebraic groups}, Graduate Texts in Mathematics, No. 21. 
Springer-Verlag, New York-Heidelberg, 1975.

\bibitem{LPS} M.W. Liebeck, C.E. Praeger and J. Saxl, Transitive subgroups of primitive permutation groups,
Special issue in honor of Helmut Wielandt.
\emph{J. Algebra} \textbf{234} (2000),  291--361. 

\bibitem{PT} M.A. Pellegrini and M.C. Tamburini Bellani, More on regular subgroups of the affine group,
\emph{Linear Algebra Appl.} \textbf{505} (2016), 126--151. 

\bibitem{Ischia} M.C. Tamburini Bellani, Some remarks on regular subgroups of the affine group,
\emph{Int. J. Group Theory} \textbf{1} (2012) 17--23.

\bibitem{T} D.E. Taylor, \emph{The geometry of the classical groups}, Sigma Series in Pure Mathematics, 9. Heldermann 
Verlag, Berlin, 1992

\end{thebibliography}
\end{document}